\documentclass{amsart}

\usepackage{amsmath,amsfonts,amsthm,amssymb,verbatim,enumerate,graphicx,setspace}
\usepackage{color}
\usepackage[usenames]{xcolor}
\usepackage[flushmargin]{footmisc}
\newcommand{\comments}[1]{}
\newtheorem{theorem}{Theorem}[section]

\newtheorem{lemma}{Lemma}[section]
\newtheorem{corollary}{Corollary}[section]

\newtheorem{remarkk}{Remark}
\numberwithin{equation}{section}
\def \isnatural {\in\mathbb{N}}

\def \iscomplex {\in\mathbb{C}}

\def \spw {spider's web}
\def \spws {spiders' webs}

\def \varname {\nu}
\newcommand{\tef}{transcendental entire function}
\newcommand\qfor{\quad\text{for }}

\makeatletter
\def\blfootnote{\xdef\@thefnmark{}\@footnotetext}
\makeatother

\begin{document}
\title[Julia and escaping set spiders' webs of positive area]{Julia and escaping set spiders' webs of positive area}
\author{D. J. Sixsmith}
\address{Department of Mathematics and Statistics \\
	 The Open University \\
   Walton Hall\\
   Milton Keynes MK7 6AA\\
   UK}
\email{david.sixsmith@open.ac.uk}
%
%
\begin{abstract}
We study the dynamics of a collection of families of {\tef}s which generalises the well-known exponential and cosine families. We show that for functions in many of these families the Julia set, the escaping set and the fast escaping set are all {\spws} of positive area. This result is unusual in that most of these functions lie outside the Eremenko-Lyubich class $\mathcal{B}$. This is also the first result on the area of a {\spw}.
\end{abstract}
\maketitle
%
%
%
\blfootnote{The author was supported by Engineering and Physical Sciences Research Council grant EP/J022160/1.}
\section{Introduction}
Suppose that $f:\mathbb{C}\rightarrow\mathbb{C}$ is a {\tef}. The \itshape Fatou set \normalfont $F(f)$ is defined as the set of points $z\iscomplex$ such that $(f^n)_{n\isnatural}$ is a normal family in a neighbourhood of $z$. Since $F(f)$ is open, it consists of at most countably many connected components, called \itshape Fatou components\normalfont. The \itshape Julia set \normalfont $J(f)$ is the complement in $\mathbb{C}$ of $F(f)$. An introduction to the properties of these sets was given in \cite{MR1216719}.

For a general {\tef} the \itshape escaping set \normalfont $$I(f) = \{z : f^n(z)\rightarrow\infty\text{ as }n\rightarrow\infty\}$$ was first studied in \cite{MR1102727}, where it was shown, for example, that $J(f) = \partial I(f)$. The set $I(f)$ now plays a key role in the study of complex dynamics. It was asked in \cite{MR1102727} whether $I(f)$ has only unbounded components, and this remains a major open problem known as \itshape Eremenko's conjecture\normalfont. 

The \itshape fast escaping set \normalfont $A(f)$ is a subset of $I(f)$. It was introduced in \cite{MR1684251}, and was defined in \cite{Rippon01102012} by
\begin{equation}
\label{Adef}
A(f) = \{z : \text{there exists } \ell \isnatural \text{ such that } |f^{n+\ell}(z)| \geq M^n(R,f), \text{ for } n \isnatural\}.
\end{equation}
Here, the \itshape maximum modulus function \normalfont $M(r,f) = \max_{|z|=r} |f(z)|,$ for $r \geq 0,$ $M^n(r,f)$ denotes repeated iteration of $M(r,f)$ with respect to the variable $r$, and $R > 0$ is such that $M(r,f) > r$, for $r \geq R$. The set $A(f)$ also now plays a key role in the study of complex dynamics, in particular in relation to partial progress on Eremenko's conjecture. We refer to \cite{Rippon01102012} for a detailed account of the properties of $A(f)$.

Following \cite{Rippon01102012}, we also define the related sets
\begin{equation}
\label{defAR}
A_R(f) = \{z : |f^n(z)| \geq M^n(R,f), \text{ for } n\isnatural\},
\end{equation}
where $R > 0$ is such that $M(r,f) > r$, for $r \geq R$.

We study the families of {\tef}s defined, for $n\isnatural$, by $$\mathcal{E}_n = \{ f : f(z) = \sum_{k=0}^{n-1} a_k \exp(\omega_n^k z), \text{ where } a_k \ne 0 \text{ for } k\in\{0,1,\cdots,n-1 \}\},$$ where $\omega_n = \exp(2\pi i/n)$ is an $n$th root of unity. We set $$\mathcal{E} = \bigcup_{n=1}^\infty \mathcal{E}_n.$$

The families $\mathcal{E}_1$ and $\mathcal{E}_2$ are both well-known. $\mathcal{E}_1$ is the exponential family $$\{ f : f(z) = \lambda \exp(z), \ \lambda \ne 0\},$$ and $\mathcal{E}_2$ is the cosine family $$\{ f : f(z) = \gamma\exp(z) + \delta\exp(-z), \ \gamma\ne 0, \ \delta\ne 0\}.$$

%
%
The first part of this paper concerns the size of $J(f) \cap A(f)$, for $f \in \mathcal{E}$. The size of the Julia sets of functions in $\mathcal{E}_1$ and $\mathcal{E}_2$ was considered by McMullen \cite{MR871679}. In particular he showed that if $f\in \mathcal{E}_1$, then $J(f)$ has Hausdorff dimension equal to $2$; see, for example, \cite{falconer} for a definition of Hausdorff dimension. McMullen also showed that though there are many functions $f \in \mathcal{E}_1$ such that $J(f)$ has area equal to zero, if $f\in \mathcal{E}_2$, then $J(f)$ has positive area. In fact, it can be seen from an analysis of the construction in McMullen's paper that if $f\in \mathcal{E}_2$, then $J(f) \cap A(f)$ has positive area. Our first main result is a generalisation of this fact to the case $n\geq 2$.
\begin{theorem}
\label{T2}
Suppose that $f \in \mathcal{E}_n$, for $n\geq 2$. Then $J(f) \cap A(f)$ has positive area.
\end{theorem}
Most papers on the size of $J(f)$ concern functions in the class $\mathcal{B}$, or, more generally, functions with a logarithmic tract; see below for definitions of these concepts. It is straightforward to show that $\mathcal{E}_1 \cup \mathcal{E}_2  \subset \mathcal{B}$. We show in Lemma~\ref{LnotinB} below that functions in $\mathcal{E}_n$, for $n\geq 3$, do not have a logarithmic tract. It follows that $\mathcal{E} \cap \mathcal{B} = \mathcal{E}_1 \cup \mathcal{E}_2$, and that Theorem~\ref{T2} is relatively unusual in that it applies to functions without a logarithmic tract. 

The \itshape Eremenko-Lyubich class\normalfont, $\mathcal{B}$, is the class of {\tef}s for which the set of singular values is bounded. There are many results on the size of $J(f)$, and its subsets, for functions in class $\mathcal{B}$. For example, it was shown in \cite{MR1357062} that if $f\in\mathcal{B}$, then $J(f)$ has Hausdorff dimension strictly greater than one. 

As noted in \cite{MR2439666}, the class $\mathcal{B}$ can be generalised to the class of functions with a \itshape logarithmic tract\normalfont; see Section~\ref{Stracts} for the definition of a logarithmic tract. For example, the result of \cite{MR1357062} was generalised to functions with a logarithmic tract in \cite{MR2480096}. McMullen's result on functions in $\mathcal{E}_2$ was strengthened to some functions with a logarithmic tract in \cite{MR2885574}. Although there are some interesting similarities between our results and those of \cite{MR2885574}, it does not seem possible to apply the approach of \cite{MR2885574} to functions without a logarithmic tract. 

An important paper which concerns the size of $J(f) \cap I(f)$ for functions $f$ which need not have a logarithmic tract is that of Bergweiler and Karpi{\'n}ska \cite{MR2609307}. We discuss the relationship between Theorem~\ref{T2} and the results in \cite{MR2609307} in Section~\ref{BK}. \\

%
%
The second part of this paper concerns the structure of $J(f) \cap A(f)$, for $f\in\mathcal{E}$. Devaney and Krych \cite{MR758892} studied the Julia set of many functions in $\mathcal{E}_1$; note \cite{MR1196102} that for functions in class $\mathcal{B}$, $I(f)$, and hence $A(f)$, is a subset of $J(f)$. They showed that the Julia set of one of these functions is a closed set consisting of an uncountable union of disjoint unbounded curves. Devaney and Tangerman \cite{MR873428} first used the name \itshape Cantor bouquet \normalfont for this structure,  and showed that there is a large class of functions, including many exponentials such as $f(z)~=~\frac{1}{4} e^z$, for which the Julia set is a Cantor bouquet. For a general study of Cantor bouquets, including a precise definition, we refer to \cite{MR2902745}.

Schleicher and Zimmer \cite{MR1956142} studied the whole of $\mathcal{E}_1$, and showed that every point in the escaping set of any function in this family lies on an unbounded curve in the escaping set. Rottenfusser and Schleicher \cite{MR2458810} showed that the same is true for functions in $\mathcal{E}_2$. Rempe, Rippon and Stallard \cite{MR2675603} showed that these facts also apply to $J(f) \cap A(f)$. 

We show that if $f \in \mathcal{E}_n$, for $n\geq 3$, then $J(f)~\cap~A(f)$ has a structure known as a {\spw}. A set $E$ is defined in \cite{Rippon01102012} as a \itshape {\spw} \normalfont if it is connected and there exists a sequence of bounded simply connected domains $(G_n)_{n\isnatural}$ such that $$\partial G_n \subset E, \ G_n \subset G_{n+1}, \text{ for } n\isnatural, \text{ and } \bigcup_{n\isnatural}~G_n~=~\mathbb{C}.$$ 

Suppose that $f$ is a {\tef}, and that $R > 0$ is such that $M(r,f) > r$, for $r \geq R$. It was shown in \cite{Rippon01102012} that if $A_R(f)$ is a {\spw}, then so are $A(f)$ and $I(f)$. In \cite[Section 8]{Rippon01102012} many examples of {\tef}s $f$ such that $A_R(f)$ is a {\spw} were given. Further examples were given in \cite{MR2917092}, and also in \cite{MR2838342} which gave the relatively simple example $$g(z) = \cos z + \cosh z,$$ for which $A_R(g)$ is a {\spw}. We observe that $g \in \mathcal{E}_4$. Our second main theorem is a generalisation of this observation.
\begin{theorem}
\label{T1}
Suppose that $f \in \mathcal{E}_n$, for $n\geq 3$, and that $R > 0$ is such that $M(r,f) > r$, for $r \geq R$. Then each of $$A_R(f), \ A(f), \ I(f), \ J(f) \cap A_R(f), \ J(f) \cap A(f), \ J(f) \cap I(f), \text{ and } J(f)$$ is a {\spw}. 
\end{theorem}
We note that Theorem~\ref{T1} cannot be extended to $n \in \{1, 2\}$ since if $f\in\mathcal{B}$, then $A_R(f)$ is not a {\spw}; see \cite[Theorem 1.8]{Rippon01102012} reproduced in Lemma~\ref{RS1} below.
 
It follows from Bergweiler and Karpi{\'n}ska's result \cite[Theorem 1.1]{MR2609307}, applied to examples in \cite{Rippon01102012} and \cite{MR2838342}, that there are {\tef}s, $f$, for which $I(f)$ and $J(f)$ are {\spws} with Hausdorff dimension equal to $2$. The following is an immediate corollary of Theorem~\ref{T2} and Theorem~\ref{T1}, and shows that there are {\tef}s for which these sets are {\spws} of positive area.
\begin{corollary}
\label{T3}
Suppose that $f \in \mathcal{E}_n$, for $n\geq 3$. Then $A(f)$, $I(f)$ and $J(f)$ are {\spws} of positive area. 
\end{corollary}
\begin{remarkk}
\normalfont
Theorem~\ref{T2} and Theorem~\ref{T1} can be applied to a more general class of {\tef}s. In particular the results of these theorems hold if $f$ is a {\tef} such that, for some $n\geq 3$, $$f(z) = \sum_{k=0}^{n-1} a_k \exp(b_k z), \text{ where } a_kb_k \ne 0 \text{ for } k \in \{0, 1, \cdots, n-1\},$$ and that $$\operatorname{arg}(b_k) < \operatorname{arg}(b_{k+1}) < \operatorname{arg}(b_k) + \pi, \ \text{ for } k\in\{0, 1, \cdots, n-2\},$$ and finally that $$\operatorname{arg}(b_0) < \operatorname{arg}(b_{n-1}) - \pi.$$ Here, we choose the value of the argument to lie in $[0, 2\pi)$. However, the proofs of these facts are slightly more complicated, and the symmetries of the {\tef}s in $\mathcal{E}$, together with the natural generalisation of the exponential and cosine families, seem to make $\mathcal{E}$ the most interesting sub-class of this more general class.
\end{remarkk}
The structure of this paper is as follows. First, in Section~\ref{BK}, we briefly discuss the paper of Bergweiler and Karpi{\'n}ska mentioned earlier. In Section~\ref{S1} we prove a sufficient condition for a point to be in the Julia set. This may be of independent interest. In Section~\ref{S1a} we show that if $f$ is a function in $\mathcal{E}_n$, for $n \geq 3$, then there are large areas of the plane in each of which $f$ behaves like a single exponential. This enables us to deduce that there is a large area of the plane in which $f$ is conformal in any square of a certain fixed side. Section~\ref{S3} concerns distortion and nonlinearity, and we use a result and construction of McMullen \cite[Proposition~3.1]{MR871679} to show that the distortion of the iterates of $f$ is bounded above in these squares. These two facts about the behaviour of $f$ in these squares enable us to estimate the dimension of the Julia set even though $f$ lies outside $\mathcal{B}$. In Section~\ref{S3a} we prove Theorem~\ref{T2}, and in Section~\ref{S2} we prove Theorem~\ref{T1}. Finally, in Section~\ref{Stracts} we give some definitions and prove a result regarding logarithmic tracts.
%
%
%
%
%
\section{Results of Bergweiler and Karpi{\'n}ska}
\label{BK}
Bergweiler and Karpi{\'n}ska \cite[Theorem 1.1]{MR2609307} showed that there exists a large class of functions, many of which are outside of the class $\mathcal{B}$, for which $J(f) \cap I(f)$ has Hausdorff dimension equal to $2$. Their proof of the following result used several ideas in a novel manner, including the Ahlfors islands theorem, a result of McMullen, and the construction of a large set of points in which the size of the logarithmic derivative is tightly constrained.
\begin{theorem}
\label{BKtheo}
Suppose that $f$ is a {\tef}, and that there exist $A, B, C, r_0 > 1$ such that
\begin{equation}
\label{BKeq}
A \log M(r, f) \leq \log M(Cr, f) \leq B \log M(r, f), \qfor r\geq r_0.
\end{equation}
Then $J(f) \cap I(f)$ has Hausdorff dimension equal to $2$. 
\end{theorem}
We first show that if $f \in \mathcal{E}$, then $f$ satisfies the hypotheses of Theorem~\ref{BKtheo}, and hence $J(f) \cap I(f)$ has Hausdorff dimension equal to $2$.
\begin{lemma}
\label{fitsBK}
Suppose that $f\in \mathcal{E}$. Then there exist $A, B, C, r_0 > 1$ such that (\ref{BKeq}) holds.
\end{lemma}
\begin{proof}
Recall that $f(z) = \sum_{k=0}^{n-1} a_k \exp(\omega_n^k z)$, for some $n\isnatural$. We note first that $$M(r, f) \leq e^r \sum_{k=0}^{n-1} |a_k|, \qfor r>0.$$ 

We claim that $$M(r,f) \geq |f(r)| \geq \frac{1}{2}|a_0| e^r, \qfor \text{large } r.$$ This is immediate for $n\in \{1, 2\}$, and follows from Lemma~\ref{ineq.lemma}, below, for $n\geq 3$. We deduce that there exist constants $\alpha_1, \alpha_2 \in\mathbb{R}$ such that  
\begin{equation}
\label{myMeq}
r + \alpha_1 \leq \log M(r, f) \leq r + \alpha_2, \qfor\text{large } r.
\end{equation}
The lemma follows easily from (\ref{myMeq}).
\end{proof}
In proving Theorem~\ref{T2}, we show that for functions in $\mathcal{E}_n$, for $n\geq 2$, the conclusion of Theorem~\ref{BKtheo} can be strengthened in two ways. Firstly that $I(f)$ can be replaced by $A(f)$, and secondly that $J(f) \cap A(f)$ has positive area rather than Hausdorff dimension $2$.\\

We also use the following result of Bergweiler and Karpi{\'n}ska \cite[Theorem 4.5]{MR2609307}.
\begin{lemma}
\label{nomconn}
Suppose that $f$ is a {\tef}, and that there exist $A, B, C, r_0 > 1$ such that (\ref{BKeq}) holds. Then $f$ has no multiply connected Fatou components.
\end{lemma}

The following is an immediate consequence of Lemma~\ref{fitsBK} and Lemma~\ref{nomconn}.
\begin{corollary}
\label{Cnomconn}
Suppose that $f\in \mathcal{E}$. Then $f$ has no multiply connected Fatou components.
\end{corollary}
%
%
%
%
\section{A sufficient condition for a point to be in the Julia set}
\label{S1}
The main result of this section shows that, in general, an escaping point with a certain orbit is either in a multiply connected Fatou component or in the Julia set. This result may well be known, but we are not aware of a reference.
\begin{theorem}
\label{TinJulia}
Suppose that $f$ is a {\tef} and that $z_0 \in~I(f)$. Set $z_n = f^n(z_0)$, for $n\isnatural$. Suppose that there exist $\lambda > 1$ and $N\geq 0$ such that 
\begin{equation}
\label{orbiteq}
f(z_{n}) \ne 0 \quad\text{ and }\quad \left|z_n \frac{f'(z_n)}{f(z_n)}\right| \geq \lambda, \qfor n\geq N.
\end{equation}
Then either $z_0$ is in a multiply connected Fatou component of $f$, or $z_0 \in J(f)$.
\end{theorem} 
For $\zeta \iscomplex$ and $ \rho > 0$, we define a disc $$B(\zeta, \rho) = \{ z : |z - \zeta| < \rho \},$$ and a circle $$C(\zeta, \rho) = \{ z : |z - \zeta| = \rho \}.$$ To prove Theorem~\ref{TinJulia}, we use the following, which follows straightforwardly from of a result of Hayman \cite[Theorem 4.13]{hayman1994multivalent}.
\begin{lemma}
\label{hay}
Suppose that $f$ is analytic in $B(z_0, r)$, and that $0 < R < r|f'(z_0)|/4$. Then there exists $R' > R$ such that $C(f(z_0), R') \subset f(B(z_0, r))$.
\end{lemma}
The following corollary is immediate, since, by \cite[Lemma 4.2]{2011arXiv1112.5103R}, the image of a simply connected Fatou component is contained in a simply connected Fatou component. This corollary can be seen as a Koebe $\frac{1}{4}$ theorem for simply connected Fatou components, with no requirement of univalence.
\begin{corollary}
\label{c1}
Suppose that $f$ is a {\tef}, that $U$ is a simply connected Fatou component of $f$, and that $U_1$ is the Fatou component containing $f(U)$. Suppose also that $z_0\in U$, that $r>0$ is such that $B(z_0, r) \subset U$, and that $f'(z_0) \ne 0$. Then $$B\left(f(z_0), \frac{r|f'(z_0)|}{4}\right) \subset U_1.$$
\end{corollary}
%
%
If $U$ is a Fatou component such that $U \cap I(f)\ne\emptyset$, then $U \subset I(f)$ by normality. We call a Fatou component in $I(f)$ \itshape escaping\normalfont. We deduce from Corollary~\ref{c1} that there is an upper bound on the absolute values of the logarithmic derivatives of the iterates of a {\tef} in a compact subset of a simply connected escaping Fatou component.
\begin{lemma}
\label{l2}
Suppose that $f$ is a {\tef}, that $U$ is a simply connected escaping Fatou component of $f$, and that $K$ is a compact subset of $U$. 
Then there exist $C= C(K)>0$ and $N=N(K)\isnatural$ such that
\begin{equation*}
\left|\frac{(f^{n})'(z)}{f^n(z)}\right| \leq C, \qfor n\geq N, z\in K.
\end{equation*}
\end{lemma}
\begin{proof}
Fix a value of $w\in J(f)$. Choose $N\isnatural$ such that $$|f^n(z)| > |w|, \qfor n\geq N, z \in K.$$ Let $\delta=\delta(K)>0$ be sufficiently small that $B(z, \delta) \subset U$, for all $z \in K$. We claim that 
\begin{equation}
\label{toprove}
\left|\frac{(f^{n})'(z)}{f^{n}(z)}\right| \leq \frac{8}{\delta}, \qfor n\geq N, z \in K.
\end{equation}
To prove this, let $n\geq N$ and $z \in K$. We may assume that $(f^n)'(z) \ne 0$, since otherwise there is nothing to prove. We apply Corollary~\ref{c1} to $f^n$ and deduce that
\begin{equation*}
B\left(f^n(z), \frac{\delta}{4}\left|\frac{(f^{n})'(z)}{f^{n}(z)}\right|\left|f^{n}(z)\right|\right) \subset F(f), \qfor n\geq N, z \in K.
\end{equation*}
Equation (\ref{toprove}) follows, since $w\notin F(f)$. This completes the proof.
\end{proof}
We now prove Theorem~\ref{TinJulia}.
\begin{proof}[Proof of Theorem~\ref{TinJulia}]
Suppose that $z_0 \in I(f)$ and that there exist $\lambda > 1$ and $N\isnatural$ such that (\ref{orbiteq}) holds. Taking a subsequence if necessary, we may assume that $N = 0$ and that $z_0 \ne 0$. Suppose also that $z_0$ is in a simply connected Fatou component of $f$.

We apply Lemma~\ref{l2} with $K = \{ z_0 \}$, and deduce that the sequence $\left(\left|\frac{(f^{n})'(z_0)}{f^n(z_0)}\right|\right)_{n\isnatural}$ is bounded above. However, this is a contradiction since, by the chain rule, $$\left|\frac{(f^{n})'(z_0)}{f^n(z_0)}\right| = \frac{1}{|z_0|}
\left|z_0\frac{f'(z_0)}{f(z_0)}\right|
\left|z_1\frac{f'(z_1)}{f(z_1)}\right|
\cdots 
\left|z_{n-1}\frac{f'(z_{n-1})}{f(z_{n-1})}\right|
\geq \frac{\lambda^n}{|z_0|}, \qfor n\isnatural.$$ 
\end{proof}
%
%
%
%
%
%
%
\section{The behaviour of $f\in \mathcal{E}_n$, for $n\geq 3$}
\label{S1a}
Let $f\in \mathcal{E}_n$, for some $n\geq 3$. Recall that $f(z) = \sum_{k=0}^{n-1} a_k \exp(\omega_n^k z)$, and $\omega_n~=~\exp(2\pi i/n)$. We construct $n$ large sets, in each of which $f$ behaves like a single exponential, and then prove several useful inequalities on the size of $f$ and its derivatives in these sets. We use these results later to construct a set $\mathcal{K} \subset J(f) \cap A(f)$, of positive area, in a similar manner to the constructions in \cite{MR1680622} and \cite{MR871679}.

We note that functions of the form considered in this paper are part of a more general class known as \itshape exponential polynomials\normalfont. See, for example, \cite{MR1501506} for an early paper on this class. However, we have not been able to identify the precise estimates we require about such functions in earlier work, and so we give all the detail necessary for a self-contained account of our results. \\

Choose a value of $\sigma$ such that 
\begin{equation}
\label{sigmadef}
0 < \sigma < \frac{1}{8\sqrt{2}}.
\end{equation} 
Fix a value of $\eta > 4/\sigma$ and note that $\eta > 8$. Fix also a value of $\tau$ sufficiently large that
\begin{equation}
\label{etaeq}
\tau \geq \frac{1}{2 \sin(\pi/n)}\log \frac{4n\eta \max \{|a_k| : 0\leq k\leq n-1\}}{\min \{|a_k| : 0\leq k\leq n-1\}} > 0.
\end{equation} 

Suppose that $\varname>0$ is large compared to $\tau$. Let $P(\varname)$ be the interior of the regular $n$-gon centred at the origin and with vertices at the points $$\frac{\varname}{\cos(\pi/n)} \exp\left(\frac{(2k+1)i\pi}{n}\right), \qfor k\in\{0,1,\cdots,n-1\}.$$ 

Define the domains 
\begin{equation}
\label{Qdef}
Q_k = \left\{z \exp\left(\frac{(1-2k)i\pi}{n}\right) : \operatorname{Re}(z) > 0, |\operatorname{Im}(z)| < \tau\right\}, \qfor k\in\{0,1,\cdots,n-1\}.
\end{equation}
 Roughly speaking, each $Q_k$ can be obtained by rotating a half-infinite horizontal strip of width $2\tau$ around the origin until a vertex of $P(\varname)$ is positioned centrally in the strip.

Set
\begin{equation}
\label{Rdef}
R(\varname) = \mathbb{C}\ \backslash\ \left(P(\varname) \cup \bigcup_{k=0}^{n-1} Q_k\right).
\end{equation}
The set $R(\varname)$ consists of $n$ unbounded simply connected components, which are arranged rotationally symmetrically. We label these $R_p(\varname)$, for $p \in \{0, 1, \cdots, n-1\}$, where $R_0(\varname)$ has unbounded intersection with the positive real axis, and $R_{p+1}(\varname)$ is obtained by rotating $R_p(\varname)$ clockwise around the origin by $2\pi/n$ radians; see Figure~\ref{casen5}. 
\begin{figure}[ht]
	\centering
	\includegraphics[width=12cm,height=9cm]{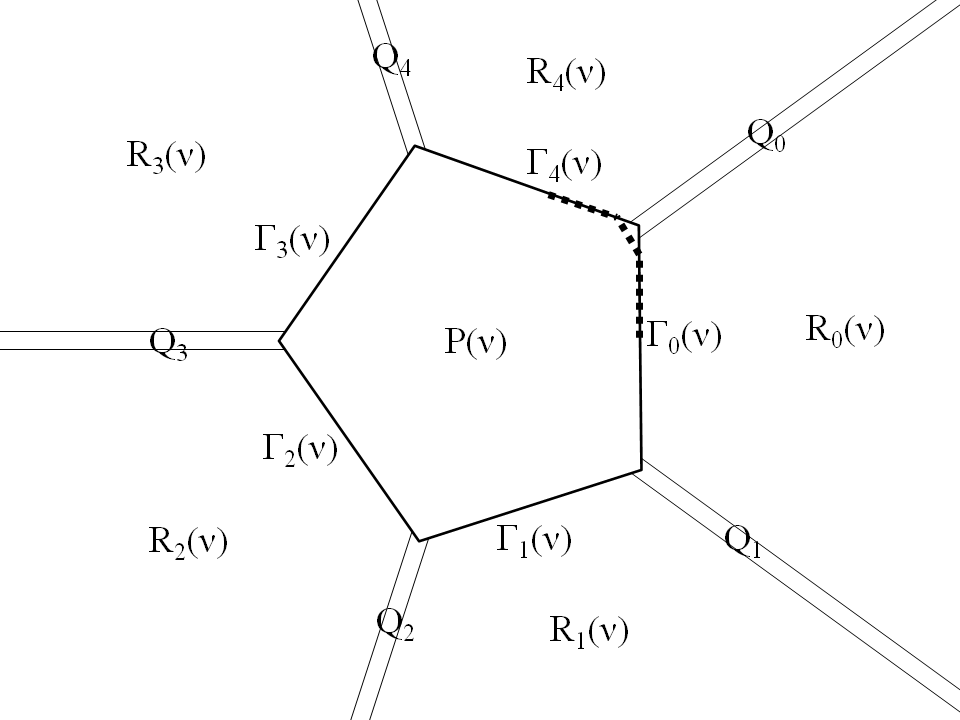}
	\caption{An illustration of the sets in the case $n=5$.}
	\label{casen5}
\end{figure}

We now prove a result which shows that if $\varname$ is sufficiently large, then $f$ behaves very like a single exponential in each component of $R(\varname)$.

Define {\tef}s $\psi_p$, for $p\in\{0, 1, \cdots, n-1\}$, by
\begin{equation}
\label{eqDp1}
\psi_p(z) = \frac{f(z)}{a_p\exp(\omega_n^p z)}  - 1.
\end{equation}

\begin{lemma}
\label{ineq.lemma}
Suppose that $n\geq3$ and that $f\in \mathcal{E}_n$. Suppose that $\eta$, $\tau$, $R_p(\varname)$, $R(\varname)$ and $\psi_p$ are as defined above, for $p\in\{0, 1, \cdots, n-1\}$. Then there exists $\varname'>0$ such that the following holds. Suppose that $\varname \geq \varname'$. Then
\begin{equation}
\label{eqDp2}
\max\{|\psi_p(z)|,|\psi_p'(z)|,|\psi_p''(z)|\} \leq 1/\eta,\qfor z \in R_p(\varname), \ p\in\{0, 1, \cdots, n-1\}.
\end{equation}
Moreover, there exists a constant $\epsilon_0\in(0,1)$, independent of $\varname$, such that, for all $z \in R(\varname)$, 
\begin{equation}
\label{uniform.expansion}
|f'(z)| > 2,
\end{equation}
\begin{equation}
\label{bounded.nonlinearity}
\left|\frac{f''(z)}{f'(z)}\right| < 2,
\end{equation}
\begin{equation}
\label{there.is.lambda}
\left|z\frac{f'(z)}{f(z)}\right| > 2,
\end{equation}
and finally
\begin{equation}
\label{there.is.epsilon}
|f(z)| > \max\{e^{\epsilon_0 \varname}, M(\epsilon_0|z|,f)\}.
\end{equation}
\end{lemma}
\begin{proof}
For $p\in\{0,1,\cdots,n-1\}$, define $\Gamma_p(\varname) = \partial P(\varname) \cap \partial R_p(\varname),$ and let $\Gamma_p'(\varname)$ be the unbounded line formed by continuing $\Gamma_p(\varname)$ to infinity in both directions. 

The following elementary observation relates the moduli of the summands in $f$ to the geometry of Figure~\ref{casen5}. Since $|\exp(z)| = e^{\operatorname{Re}(z)}$, the value of $|a_0 \exp(z)|$ is determined only by the signed perpendicular distance from $z$ to $\Gamma_0'(\varname)$. By an obvious argument from symmetry, the value of $|a_p \exp(\omega_n^p z)|$ is similarly determined by the signed perpendicular distance from $z$ to $\Gamma_p'(\varname)$, for $p\in\{1, 2, \cdots, n-1\}$.

Suppose that $z \in R_0(\varname)$. We claim that, for all sufficiently large values of $\varname$, we have that 
\begin{equation}
\label{a0like}
|a_0 \exp(z)| \geq 4n\eta |a_p \exp(\omega_n^p z)|, \qfor p \in \{1, 2, 3, \cdots, n-1\}.
\end{equation}
It follows immediately from the observation above that we may choose $\varname'>0$, large compared to $\tau$, and sufficiently large that $$|a_0 \exp(z)| \geq 4n\eta |a_p \exp(\omega_n^p z)|, \qfor z \in R_0(\varname), \varname \geq \varname' \text{ and } p \in \{2, 3, \cdots, n-2\}.$$ It remains to consider the cases $p=1$ and $p=n-1$.

Suppose then that $p=1$. We refer to Figure~\ref{newfig}, in which $z$ is at a point labeled $D$. Here the lines $\Gamma_0'(\varname)$ and $\Gamma_{1}'(\varname)$ are shown as solid lines. We are interested, by the observations above, in the difference between the perpendicular distance from $z$ to $\Gamma_{1}'(\varname)$ -- which is the distance $CD$ on the diagram -- and the perpendicular distance from $z$ to $\Gamma_0'(\varname)$ -- which is the distance $AD$ on the diagram. We also show, in dashed, the line through $z$ parallel to the sides of $Q_1$, and the intersections of this line with $\Gamma_0'(\varname)$, at the point $B$, and $\Gamma_{1}'(\varname)$, at the point $X$. The angles $\angle ADB$ and $\angle CDB$ are quickly seen to be equal to $\pi/n$. A straightforward geometric exercise shows that $BX \geq 2\tau\tan (\pi/n)$. Hence $$AD - CD = BD \cos (\pi/n) - XD \cos (\pi/n) = BX \cos (\pi/n) \geq 2\tau \sin (\pi/n).$$

\begin{figure}[ht]
	\centering
	\includegraphics[width=12cm,height=9cm]{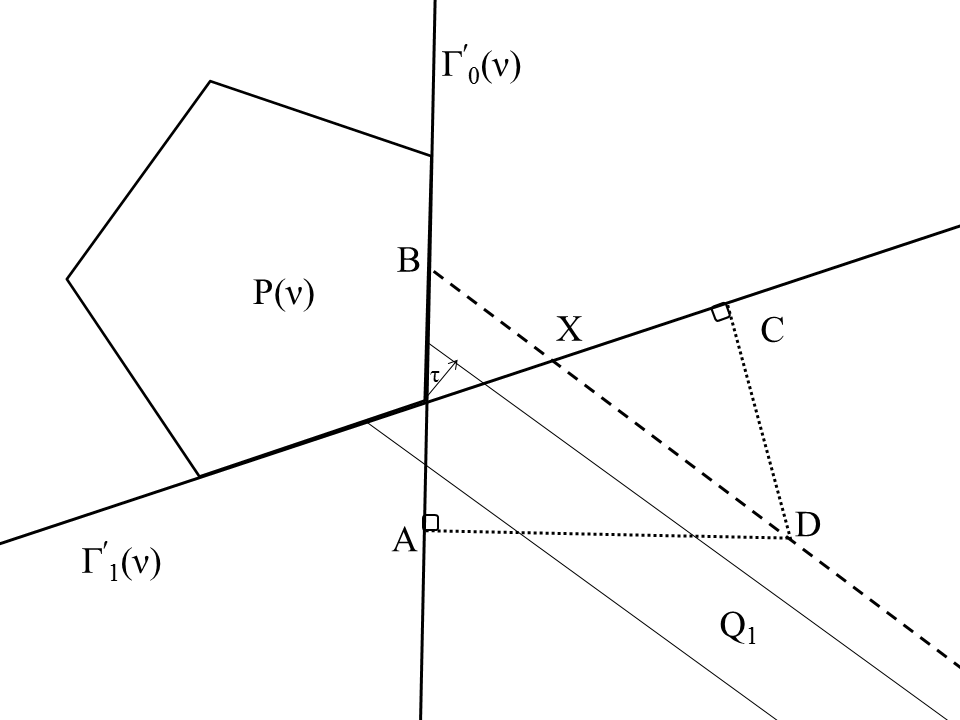}
	\caption{An illustration of the geometry in the case $n=5$ and $p=1$.}
	\label{newfig}
\end{figure}

The estimate (\ref{a0like}) in the case $p=1$ follows because, by (\ref{etaeq}), we deduce that $$\frac{|\exp(z)|}{|\exp(\omega_n^{n-1} z)|} \geq \frac{4n\eta \max \{|a_k| : 0\leq k\leq n-1\}}{\min \{|a_k| : 0\leq k\leq n-1\}}.$$ The case $p =n-1$ is very similar. This completes the proof of (\ref{a0like}).\\

We deduce that
\begin{equation}
\label{eqa0}
|a_0 \exp(z)| \geq 4\eta \sum_{k=1}^{n-1} |a_k \exp(\omega_n^k z)|, \qfor z \in R_0(\varname).
\end{equation}

It follows from (\ref{eqa0}) that $$|\psi_0(z)| = \left|\frac{\sum_{k=1}^{n-1} a_k \exp(\omega_n^k z)}{a_0\exp(z)}\right| \leq \frac{\sum_{k=1}^{n-1} |a_k \exp(\omega_n^k z)|}{|a_0\exp(z)|} \leq \frac{1}{4\eta},\qfor z \in R_0(\varname).$$ Similarly, by differentiating, $$|\psi_0'(z)| = \left|\frac{\sum_{k=1}^{n-1} a_k (\omega_n^k-1) \exp(\omega_n^k z)}{a_0\exp(z)}\right| \leq \frac{\sum_{k=1}^{n-1} |2a_k \exp(\omega_n^k z)|}{|a_0\exp(z)|} \leq \frac{1}{2\eta},\qfor z \in R_0(\varname),$$ and finally, differentiating again, $$|\psi_0''(z)| = \left|\frac{\sum_{k=1}^{n-1} a_k (\omega_n^k-1)^2 \exp(\omega_n^k z)}{a_0\exp(z)}\right| \leq \frac{1}{\eta},\qfor z \in R_0(\varname).$$ This proves (\ref{eqDp2}) for $z \in R_0(\varname)$, and (\ref{eqDp2}) follows by similar arguments in the domains $R_p(\varname)$, for $p \in \{1, 2, \cdots, n - 1\}$, choosing $\varname'$ larger if necessary. \\

For the rest of the lemma, we may assume for simplicity that $z \in R_0(\varname)$ and $\varname \geq \varname'$, since the result for $z$ in another component of $R(\varname)$ follows similarly. From (\ref{eqDp1}) we deduce that $$f(z) = a_0e^z(1 + \psi_0(z)),$$ 
\begin{equation}
\label{fdasheq}
f'(z) = a_0 e^z(1 + \psi_0(z) + \psi_0'(z)),
\end{equation}
and $$f''(z) = a_0 e^z(1 + \psi_0(z) +  2 \psi_0'(z) + \psi_0''(z)).$$

Hence, taking $\varname'$ larger than the value taken earlier, if necessary, it follows from the choice of $\eta$ and from (\ref{eqDp2}) that if $z \in R_0(\varname)$ and $\varname \geq \varname'$, then $$|f'(z)| > \frac{|a_0 e^z|}{2}> 2,$$ $$\frac{|f''(z)|}{|f'(z)|} = \frac{|1 + \psi_0(z) +  2 \psi_0'(z) + \psi_0''(z)|}{|1 + \psi_0(z) + \psi_0'(z)|} < \frac{1 + 4/\eta}{1 - 2/\eta} < 2,$$ $$\left|z\frac{f'(z)}{f(z)}\right| = |z| \frac{|1 + \psi_0(z) + \psi_0'(z)|}{|1 + \psi_0(z)|} > |z|\frac{1-2/\eta}{1+1/\eta} > 2,$$ and $$|f(z)| > \frac{|a_0e^z|}{2} = \frac{|a_0| e^{\operatorname{Re}(z)}}{2} > e^{\frac{1}{2}\varname}.$$  

Equations (\ref{uniform.expansion}), (\ref{bounded.nonlinearity}), (\ref{there.is.lambda}) and the first part of the maximum in (\ref{there.is.epsilon}) follow. For the second part of the maximum in (\ref{there.is.epsilon}), we may suppose that $$0 < \epsilon_0 < \frac{1}{2} \cos \frac{\pi}{n},$$ in which case $\epsilon_0|z| \leq \frac{1}{2} \operatorname{Re}(z)$. We deduce that $$M(\epsilon_0|z|,f) \leq \sum_{k=0}^{n-1} |a_k| e^{\epsilon_0|z|}  \leq e^{\frac{1}{2}\operatorname{Re}(z)} \sum_{k=0}^{n-1} |a_k|.$$ The result follows, once again taking $\varname'$ larger if necessary.
\end{proof}
We also require the following lemma which is a simplified version of \cite[Theorem~4.24]{MR565886}.
\begin{lemma}
\label{newlemma}
Suppose that $B\subset\mathbb{C}$ is a square of side $s$, that $g:B\to\mathbb{C}$ is analytic, and that $g'(z) \ne 0$, for $z \in B$. Suppose also that $$s \sup_{z\in B} \left|\frac{g''(z)}{g'(z)}\right| \leq \frac{1}{\sqrt{2}s + 1}.$$ Then $g$ is conformal in $B$.
\end{lemma}
The following corollary of these results is central to the proof of Theorem~\ref{T2}.
\begin{lemma}
\label{uvlemma}
Suppose that $n\geq3$ and that $f\in \mathcal{E}_n$. Suppose that $B\subset R(\varname')$ is a square of side $\sigma$, where $R(\varname')$ is as defined in (\ref{Rdef}), $\varname'$ is as defined in Lemma~\ref{ineq.lemma}, and $\sigma$ is as defined in (\ref{sigmadef}).  Then $f$ is conformal in $B$.
\end{lemma}
\begin{proof}
This follows from (\ref{sigmadef}) and (\ref{bounded.nonlinearity}), and from Lemma~\ref{newlemma}.
\end{proof}
%
%
%
%
%
%
%
\section{Distortion and nonlinearity}
\label{S3}
In this section we give some preliminary definitions and results that will be used in the proof of Theorem~\ref{T2}. Suppose that $D$ is a bounded subset of $\mathbb{C}$, and that $f$ is a map which is analytic in a neighbourhood of $D$. We say that $f$ has \itshape bounded distortion \normalfont on $D$ if there exist constants $c, \ C > 0$, depending only on $f$, such that 
\begin{equation}
\label{disteq}
c <\frac{|f(x) - f(y)|}{|x - y|} < C, \qfor x,y \in D, x \ne y.
\end{equation}
We define the \itshape distortion \normalfont of $f$ in $D$ by $$L(f|_D) = \inf\{ C/c  : (c,\ C) \text{ satisfies } (\ref{disteq}) \}.$$ It is well-known that if $f$ and $g$ are {\tef}s, then
\begin{equation}
\label{inv.dist}
L(f|_D) = L(f^{-1}|_{f(D)})
\end{equation} 
\begin{equation}
\label{comp.dist}
L(f \circ g|_D) \leq L(f|_{g(D)})L(g|_D)
\end{equation} 
and finally
\begin{equation}
\label{distortioneq}
\frac{\operatorname{area}(f(A)\cap f(D))}{\operatorname{area}(f(D))} \leq L(f|_D)^2 \frac{\operatorname{area}(A \cap D)}{\operatorname{area}(D)}, \qfor A \subset D.
\end{equation}
Here $\operatorname{area}$ denotes plane Lebesgue measure.

Following McMullen \cite{MR871679}, we define the \itshape nonlinearity \normalfont of $f$ in $D$ by $$N(f|_D) = \left(\sup_{z\in D} \left|\frac{f''(z)}{f'(z)}\right|\right) \operatorname{diam}(D),$$ where $\operatorname{diam}$ denotes Euclidean diameter. We use the following lemma \cite[Lemma~2.1]{MR2367099}.
\begin{lemma}
\label{schubert}
Suppose that $f$ is a {\tef}. Suppose also that $D$ is a square such that $N(f|_D) < \frac{1}{4}$ and such that $f$ is conformal in a neighbourhood of $D$. Then $L(f|_D) \leq 1 + 8 N(f|_D)$.
\end{lemma}
We observe that \cite{MR871679} uses a similar result but with $L(f|_D) \leq 1 + O(N(f|_D))$, which would also be sufficient for our purposes.

We require the following result, which is a detailed version of \cite[Proposition~3.1]{MR871679}. Since this result is central to our work, and the proof in \cite{MR871679} is relatively brief, we give complete details.
\begin{lemma}
\label{distortion.lemma}
Suppose that $f$ is a {\tef}, and there exists a set $U\subset\mathbb{C}$ and constants $\alpha>1$ and $M > 0$ such that
\begin{equation}
\label{unifexp}
|f'(z)| > \alpha \quad\text{ and }\quad \left|\frac{f''(z)}{f'(z)}\right| < M, \qfor z \in U.
\end{equation}
Suppose also that there exists $s \in (0, (4\sqrt{2}M)^{-1})$ such that if $B \subset U$ is a square of side $s$, then $f$ is conformal in a neighbourhood of $B$. Suppose finally that $(B_n)_{n\isnatural}$ is a sequence of squares of side $s$, such that $$B_n \subset U \ \text{ and } \ B_{n+1} \subset f(B_n), \qfor n\isnatural.$$ For $n\isnatural$, let $\phi_n$ be the inverse branch of $f$ which maps $f(B_n)$ to $B_n$, and set $D_n = \phi_1 \circ \phi_2 \circ \cdots \circ \phi_{n}(f(B_n))$. Then there exists $\lambda=\lambda(M, s, \alpha)>0$ such that $$L(f^n|_{D_n})\leq \lambda.$$
\end{lemma}
In particular, we note that $\lambda$ is independent of both $n$ and the sequence $(B_n)_{n\isnatural}$.
\begin{proof}[Proof of Lemma~\ref{distortion.lemma}]
For $n\isnatural$ and $m\in\{0,1,\cdots,n-1\}$ define $$D_{n, m} = \phi_{n-m} \circ \phi_{n-m+1} \circ  \cdots \circ \phi_n (f(B_n)).$$

It follows from (\ref{unifexp}) that $\operatorname{diam} (D_{n,m}) \leq s \sqrt{2}\alpha^{-m}$. Hence, by  (\ref{unifexp}), we deduce that $$N(f|_{D_{n,m}}) \leq Ms \sqrt{2}\alpha^{-m} < \frac{1}{4}, \qfor m\in\{0,1,\cdots,n-1\}.$$ Thus, by Lemma~\ref{schubert}, $$L(f|_{D_{n,m}}) \leq 1 + 8 M s \sqrt{2}\alpha^{-m}, \qfor m\in\{0,1,\cdots,n-1\}.$$ Hence, by (\ref{comp.dist}), we have that 
\begin{align*}
L(f^n|_{D_n})    &\leq L(f|_{D_{n,n-1}}) L(f|_{D_{n,n-2}}) \ldots L(f|_{D_{n,0}})  \\
                 &\leq \prod_{m=0}^{n-1} (1 + 8M s \sqrt{2}\alpha^{-m}) \\
                 &\leq \prod_{m=0}^\infty (1 + 8M s \sqrt{2} \alpha^{-m}).
\end{align*}                 
\end{proof}
%
%
%
%
%
%
%
%
%
\section{Proof of Theorem~\ref{T2}}
\label{S3a}
%
%
We require the following alternative characterisation of $A(f)$ \cite[Theorem~2.7]{Rippon01102012}. Here we define $\mu_\epsilon(r) = M(\epsilon r,f)$, for $r > 0$ and $\epsilon>0.$ 
\begin{lemma}
\label{otherA}
Suppose that $f$ is a {\tef} and that $\epsilon > 0$. Suppose also that $R > 0$ is sufficiently large that $\mu_\epsilon(r) > r$, for $r \geq R$. Then
\begin{equation*}
A(f) = \{z : \text{there exists } \ell \isnatural \text{ such that } |f^{n+\ell}(z)|\geq \mu_\epsilon^n(R), \text{ for } n\isnatural \}.
\end{equation*}
\end{lemma} 
\begin{proof}[Proof of Theorem~\ref{T2}]
We may suppose that $f \in \mathcal{E}_n$, for some $n \geq 3$. We show that there exists a set $\mathcal{K} \subset J(f) \cap A(f)$ such that $\mathcal{K}$ has positive area. \\

%
%
We first define some variables required to start the construction. Let $\epsilon_0$ be the constant from (\ref{there.is.epsilon}) and define $$\alpha(r) = \frac{1}{2} e^{\epsilon_0 r}, \qfor r>0.$$ 

Choose $\varname_0$ sufficiently large that the following all hold;
\begin{itemize}
\item $\varname_0 \geq \varname'$, where $\varname'$ is the constant from Lemma~\ref{ineq.lemma};
\item $\alpha(\varname_0) > \varname_0$;
\item $\frac{1}{2}\sigma e^{\varname_0} \min\{|a_k|:0\leq k\leq n-1\}$ is large compared to both $\sigma$ and $\tau$, where $\sigma$ is the constant defined in (\ref{sigmadef}) and $\tau$ is the constant defined in (\ref{etaeq});
\item $\mu_{\epsilon_0}(r) > r,$ for $r \geq \varname_0.$
\end{itemize}

Define 
\begin{equation}
\label{thek}
\varname_{k} = \alpha^k(\varname_{0}), \qfor k\isnatural.
\end{equation}

%
%
We now construct the set $\mathcal{K}$. We first pack the complex plane with disjoint squares of side $\sigma$ -- which we refer to as \itshape boxes \normalfont --  by defining $$B_{m,m'} = \{ z : m\sigma < \operatorname{Re}(z) < (m+1)\sigma, m'\sigma < \operatorname{Im}(z) < (m'+1)\sigma\}, \qfor m, m' \in \mathbb{Z}.$$

Recall that, for $\varname>0$, $R(\varname)$ is defined in (\ref{Rdef}) and $P(\varname)$ is defined following (\ref{etaeq}); see Figure~\ref{casen5}. Choose $m_0, m'_0 \in \mathbb{Z}$ such that $B_{m_0, m'_0}\subset R(\varname_0)$, and set $K_0 = B_{m_0, m'_0}$. We define inductively a sequence of collections of disjoint subsets of $K_0$ as follows;
\begin{itemize}
\item $\mathcal{K}_0 = \{K_0\}$,
\item $\mathcal{K}_n$ consists of the connected sets $K_n$ satisfying the following conditions:
    \begin{enumerate}[(i)]
        \item there exist $m, m' \in \mathbb{Z}$ such that $f^n(K_n) = B_{m, m'}$ and $B_{m, m'} \subset R(\varname_n)$.
        \item $K_n \subset K_{n-1}$ for some $K_{n-1} \in \mathcal{K}_{n-1}$.
    \end{enumerate} 
\end{itemize}
We complete the construction by setting $$\widetilde{\mathcal{K}}_n = \bigcup_{K \in \mathcal{K}_n} K, \qfor n\in\{0,1,\cdots\}, \quad\text{ and }\quad \mathcal{K} = \bigcap_{n=0}^\infty \widetilde{\mathcal{K}}_n.$$

We next show that these collections are non-empty. First, we need to understand the size of the image of a box. Suppose that $B$ is a box such that $B \subset R_0(\varname)$, where $\varname \geq \varname_0$. If $z \in B$, then, by (\ref{eqDp1}) and (\ref{eqDp2}), we have $$|\arg (f(z)) - \arg (a_0 e^z)| \leq 1/\eta,$$ for some branch of the argument defined in a neighbourhood of $f(z)$, and $$|a_0 e^z|\left(1 - 1/\eta\right) \leq |f(z)| \leq |a_0 e^z|\left(1 + 1/\eta\right).$$ Hence, by symmetry, if $B$ is a box such that $B \subset R(\varname)$, where $\varname \geq \varname_0$, then $f(B)$ contains a curvilinear square of side at least $\frac{1}{2}\sigma e^{\varname} \min\{|a_k|:0\leq k\leq n-1\}$.

Suppose then that $n\isnatural$, and that $K_{n-1}\in\mathcal{K}_{n-1}$. We note two facts about the set $f^{n}(K_{n-1})$; see Figure~\ref{fig.x}. Firstly, since $f^{n-1}(K_{n-1})$ is a box contained in $R(\varname_{n-1})$, it follows from the discussion above that $f^{n}(K_{n-1})$ contains a curvilinear square, of side at least $\frac{1}{2}\sigma e^{\varname_{n-1}} \min\{|a_k|:0\leq k\leq n-1\}$. This value, by the choice of $\varname_0$, is large compared to both $\tau$ and $\sigma$. Hence the set $f^{n}(K_{n-1})$ is large compared to the size of the boxes, and also large compared to the strips $Q_k$, $k~\in~\{0,1,\cdots,n-~1\}$, which are of fixed width $2\tau$. Secondly, we note that, by (\ref{there.is.epsilon}) and (\ref{thek}), and since $|w| < 2\varname$, for $w\in P(\varname)$, we have that
\begin{equation*}
f(z) \in \mathbb{C} \ \backslash\ P(\varname_{k+1}), \qfor z\in R(\varname_k), \ k\isnatural.
\end{equation*} 
We deduce that $f^{n}(K_{n-1}) \subset \mathbb{C}\backslash P(\varname_n)$. It follows from these two facts, by induction, that $\mathcal{K}_{n}$ is non-empty for $n\isnatural$. \\

We next claim that 
\begin{equation}
\label{areaeq}
\operatorname{area}(f^{n}(K_{n-1}\backslash\widetilde{\mathcal{K}}_{n})) = O(e^{\varname_{n-1}}).
\end{equation}
The set $f^{n}(K_{n-1})$ is the image of a box, and $f^{n}(K_{n-1}\backslash\widetilde{\mathcal{K}}_{n})$ consists of all the points in this set which do not lie in a box which itself is contained in $f^{n}(K_{n-1})$. It follows that $f^{n}(K_{n-1}\backslash\widetilde{\mathcal{K}}_{n})$ consists of the union of three sets. The first is the set of points of $f^{n}(K_{n-1})$ which lie on the boundary of a box; this set has area zero. The second is the set of points of $f^{n}(K_{n-1})$ (if any) which lie in a box which intersects with one of the strips $Q_k$; the area of this set is easy to estimate as $O(e^{\varname_{n-1}})$. The third is contained in the set of points of $f^{n}(K_{n-1})$ which lie at a distance less than or equal to $\sigma\sqrt{2}$ from the boundary of $f^{n}(K_{n-1})$; it follows from (\ref{eqDp2}) and (\ref{fdasheq}) that the length of the boundary of $f^{n}(K_{n-1})$ is also $O(e^{\varname_{n-1}})$. This completes the proof of (\ref{areaeq}). \\

\begin{figure}[ht]
	\centering
	\includegraphics[width=12cm,height=9cm]{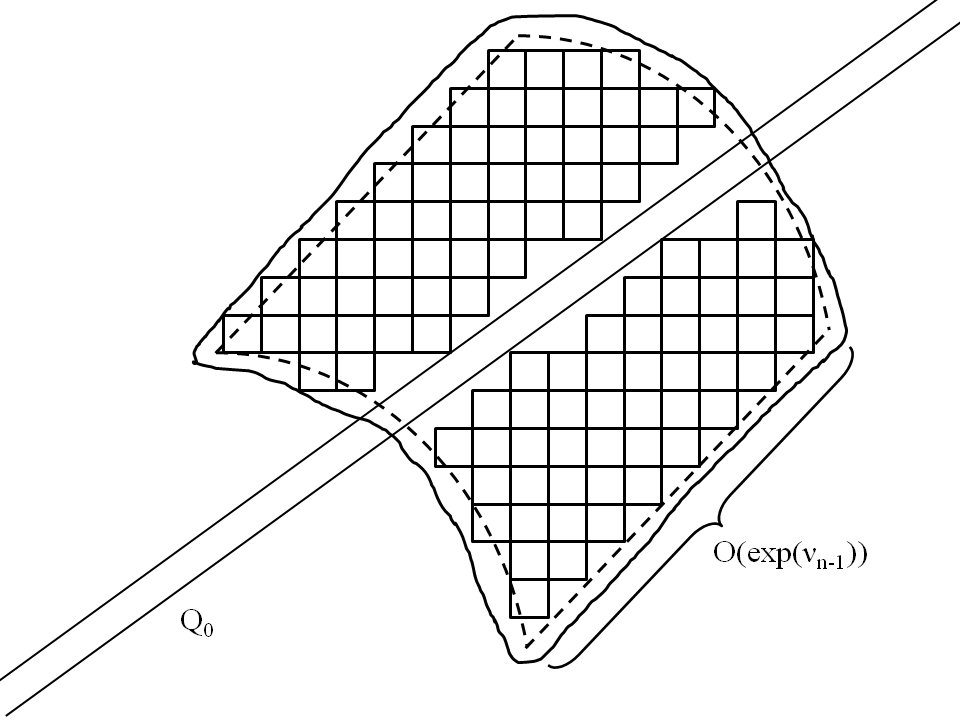}
	\caption{$f^{n}(K_{n-1})$, where $K_{n-1} \in \mathcal{K}_{n-1}$, shown with solid boundary and packed with boxes which belong to $f^{n}(\mathcal{K}_{n})$. Note that $f^{n}(K_{n-1})$ contains a slightly smaller curvilinear square -- shown with a dashed boundary -- which has side $O(\exp(\varname_{n-1}))$.}
  \label{fig.x}
\end{figure}  

Next we show that $\mathcal{K}$ has positive area. Choose $n\isnatural,$ and $K_{n-1} \in \mathcal{K}_{n-1}$. It follows from (\ref{sigmadef}), (\ref{uniform.expansion}), (\ref{bounded.nonlinearity}) and Lemma~\ref{uvlemma} that we can apply Lemma~\ref{distortion.lemma}, with $U = R(\varname_0)$, $s = \sigma$ and $\alpha = M = 2$. We deduce that the distortion of $f^{n}$ on $K_{n-1}$ is bounded independently of $n$ and $K_{n-1}$. Hence, by (\ref{distortioneq}), we have that $$\frac{\operatorname{area}(K_{n-1}\cap\widetilde{\mathcal{K}}_{n})}{\operatorname{area}(K_{n-1})} = 1- \frac{\operatorname{area}(K_{n-1}\backslash\widetilde{\mathcal{K}}_{n})}{\operatorname{area}(K_{n-1})}  \geq 1 - O\left(\frac{\operatorname{area}(f^{n}(K_{n-1}\backslash\widetilde{\mathcal{K}}_{n}))}{\operatorname{area}(f^{n}(K_{n-1}))}\right).$$

Now, by (\ref{areaeq}) $$\frac{\operatorname{area}(f^{n}(K_{n-1}\backslash\widetilde{\mathcal{K}}_{n}))}{\operatorname{area}(f^{n}(K_{n-1}))} = O(e^{-\varname_{n-1}}).$$

Making $\varname_0$ larger, if necessary, we can assume that $\sum_{n=1}^\infty e^{-\varname_{n-1}}$ is arbitrarily small. It follows that there exists $\Delta > 0$ such that $$\frac{\operatorname{area}(\mathcal{K})}{\operatorname{area} K_0} \geq \prod_{n=1}^\infty \left(1 - O\left(e^{-\varname_{n-1}}\right)\right) \geq \Delta,$$ and so $\mathcal{K}$ has positive area, as required.\\

Finally we show that $\mathcal{K} \subset J(f) \cap A(f)$. Suppose that $z \in \mathcal{K}$. It follows from (\ref{there.is.epsilon}), and by construction, that $|f^n(z)| \geq \mu_{\epsilon_0}(\varname_0)$, for $n\isnatural$. Hence, by Lemma~\ref{otherA}, and by choice of $\varname_0$, we have that $z \in A(f)$. 

It follows from (\ref{there.is.lambda}) that we can apply Theorem~\ref{TinJulia}, with $z_0 = z$, to obtain that either $z$ is in a multiply connected Fatou component of $f$, or $z \in J(f)$. However, by Corollary~\ref{Cnomconn}, $f$ has no multiply connected Fatou components, and so $z \in J(f)$. This completes the proof of Theorem~\ref{T2}.
\end{proof}
%
%
%
%
%
%
%
\section{Proof of Theorem~\ref{T1}}
\label{S2}
%
%
For the proof of Theorem~\ref{T1} we require the following \cite[Theorem 8.1]{Rippon01102012}.
\begin{theorem}
\label{webtheorem}
Let $f$ be a {\tef} and let $R > 0$ be such that $M(r,f) > r$, for $r \geq R$. Then $A_R(f)$ is a {\spw} if and only if there exists a sequence $(G_n)_{n \geq 0}$ of bounded simply connected domains such that, for all $n \geq 0$,
\begin{equation} \label{webtheorema} G_n \supset  B(0, M^n(R, f)) \end{equation}
and
\begin{equation}  \label{webtheoremb} G_{n+1} \text { is contained in a bounded component of } \mathbb{C}\ \backslash\ f(\partial G_n). \end{equation}
\end{theorem}
\begin{proof}[Proof of Theorem~\ref{T1}]
Recall again that $f(z) = \sum_{k=0}^{n-1} a_k \exp(\omega_n^k z)$, where $n\geq 3$, and $\omega_n~=~\exp(2\pi i/n)$. We first prove that $A_R(f)$ is a {\spw} by giving an explicit construction of the sequence of domains $(G_n)_{n\isnatural}$ given in the statement of Theorem~\ref{webtheorem}. The proof of the rest of the theorem then follows quickly.

Recall Figure~\ref{casen5} for an illustration of the various sets defined earlier. For large values of $\varname$, the domain $P(\varname)$, defined in Section~\ref{S1a}, is almost a candidate for a domain $G_n$, for some $n\isnatural$. However, $|f(z)|$ can be small for values of $z$ close to the vertices of $P(\varname)$. We need to modify $P(\varname)$ to a slightly smaller domain, $P'(\varname)$, in such a way that if $z$ is on the boundary of $P'(\varname)$, then $|f(z)|$ is, in some sense, large. Roughly speaking, we define the boundary of $P'(\varname)$ by `cutting off' the vertices of $P(\varname)$. 

We now explain the first part of this construction, which is illustrated by the dotted line in Figure~\ref{casen5}. Set $c_0 = \log (a_{n-1}/a_0)$, where $\log$ is any branch of the logarithm, and define a {\tef}
\begin{equation}
\label{eqg}
g_0(z) = a_0 \exp(z) + a_{n-1} \exp(\omega_n^{n-1} z) = a_0 \exp(z) (1 + \exp((\omega_n^{n-1} - 1)z + c_0)).
\end{equation}

Define also the lines 
\begin{equation}
\label{eqLm}
L_{0,m} = \{z : \operatorname{Im} ((\omega_n^{n-1} - 1)z + c_0) = 2m\pi\}, \qfor m\in\mathbb{Z}.
\end{equation}
The reason for this choice of line is as follows. If $z \in Q_0$ is of large modulus, then $f(z)$ is very close to $g_0(z)$. On these lines the two exponentials which make up $g_0(z)$ have the same argument. 

It can readily be shown that $$L_{0,m} = \left\{ z : \operatorname{Im}(z) = -\cot (\pi/n) \operatorname{Re}(z) + \frac{\operatorname{Im}(c_0) - 2m\pi}{2\sin^2 (\pi/n)}\right\}, \qfor m\isnatural.$$ It follows that these lines are perpendicular to the boundary of $Q_0$. Note that, by (\ref{eqg}) and (\ref{eqLm}),
\begin{equation}
\label{geq}
|g_0(z)| \geq |a_0 \exp(z)|, \qfor z \in L_{0,m}, \ m\in\mathbb{Z}.
\end{equation}

For large values of $\varname$, there are many values of $m\in\mathbb{Z}$ such that $L_{0,m}$ intersects both the line segments $\Gamma_0(\varname)$ and $\Gamma_{n-1}(\varname)$. Let $p_0\in\mathbb{Z}$ be the least of the values such that both $\Gamma_0(\varname) \cap L_{0,p_0} \not\subset Q_0$ and $\Gamma_{n-1}(\varname) \cap L_{0,p_0} \not\subset Q_0$. Let $\Gamma_0(\varname) \cap L_{0,p_0} = \{ z_0\}$ and let $\Gamma_{n-1}(\varname) \cap L_{0,p_0} = \{ z_0'\}$. (We remark that choosing a small value of $p_0$ ensures that $z_0$ has large imaginary part).
 
The first part of the boundary of $P'(\varname)$ is made up of the union of three line segments. The first is the segment of $\Gamma_0(\varname)$ from the midpoint of $\Gamma_0(\varname)$ to $z_0$. The second is the segment of $L_{0,p_0}$ from $z_0$ to $z_0'$. The third is the segment of $\Gamma_{n-1}(\varname)$ from $z_0'$ to the midpoint of $\Gamma_{n-1}(\varname)$. 

This is the first part of the construction of the boundary of $P'(\varname)$. The remainder of the boundary of $P'(\varname)$ is completed by repeating this process $n$ times, once again using arguments from symmetry. \\

We now consider the value of the modulus of $f$ on the boundary of $P'(\varname)$, where we assume that $\varname$ is sufficiently large for the comments after (\ref{geq}) to hold. We also assume that $\varname \geq \varname'$, where $\varname'$ is defined in Lemma~\ref{ineq.lemma}, and that $\varname$ is sufficiently large that 
\begin{equation}
\label{modzbig}
\varname \leq |z| \leq 2\varname, \qfor z \in \partial P'(\varname).
\end{equation}

Suppose first that $z\in \partial P'(\varname) \cap \Gamma_0(\varname)$. Then, by (\ref{eqDp1}) and (\ref{eqDp2}), we have that 
\begin{equation}
\label{eqA}
|f(z)| \geq \frac{1}{2}|a_0| e^\varname, \qfor z\in \partial P'(\varname) \cap \Gamma_0(\varname).
\end{equation}

Suppose next that $z\in\partial P'(\varname) \cap L_{0,p_0}$. Here we have that $|a_0 \exp(z)|$ may be comparable to $|a_{n-1} \exp(\omega_n^{n-1} z)|$, but all other terms in the summand of $f$ are of negligible modulus. Hence, by (\ref{geq}), there exists $\kappa = \kappa(\tau, n) > 0$ such that 
\begin{equation}
\label{eqB}
|f(z)| \geq \frac{1}{2} |g_0(z)| \geq \frac{1}{2} |a_0| e^{\operatorname{Re}(z)} \geq \frac{1}{2}|a_0| e^{-\kappa} e^\varname, \qfor z\in\partial P'(\varname) \cap L_{0,p_0}.
\end{equation}
(In fact, it follows from an elementary geometric argument, that we may take $\kappa = 2\pi \cot (\pi/n) + 2 \tau \sin (\pi/n)$.)

We deduce from (\ref{eqA}) and (\ref{eqB}), and from considerations of symmetry, that there exists a constant $\epsilon' > 0$, such that, for all sufficiently large values of $\varname$, 
\begin{equation}
\label{eqC}
|f(z)| > \epsilon' \ e^\varname, \qfor z \in \partial P'(\varname). 
\end{equation}

Define a real-valued function 
\begin{equation}
\label{eqbeta}
\beta(r) = \frac{\epsilon'}{2} e^{r}, \qfor r > 0.
\end{equation}
It is straightforward to see that there exists $0 < \delta < 1$ such that $\beta(r) \geq \delta M(r,f)$, for $r > 0$. Choose $R>0$ sufficiently large that $M(r,f) > r$, for $r \geq R$. 

It is well-known that 
\begin{equation*}
\frac{\log M(r,f)}{\log r} \rightarrow\infty \text{ as } r\rightarrow\infty,
\end{equation*}
and also (see, for example, \cite{Rippon01102012}) that if $k > 1$, then 
\begin{equation*}
\frac{M(kr,f)}{M(r,f)} \rightarrow\infty \text{ as } r\rightarrow\infty.
\end{equation*}
It follows that we may assume that $R$ is sufficiently large that $$\delta M(r,f) \geq \frac{1}{\delta} M(\delta r,f) \geq r, \qfor r \geq R.$$

We choose a value of $\varname$ sufficiently large for previous estimates to hold, and such that $\varname \geq R/\delta$. We deduce that
\begin{equation}
\label{beta.is.nice.eq}
\beta^n(\varname) \geq M^n(R, f), \qfor n\isnatural.
\end{equation} 

We now define the domains $(G_n)_{n \geq 0}$ in the statement of Theorem~\ref{webtheorem}. Let $$G_n = P'(\beta^{n}(\varname)), \qfor n\geq 0.$$ Equation (\ref{webtheorema}) holds by (\ref{modzbig}) and (\ref{beta.is.nice.eq}). Equation (\ref{webtheoremb}) holds by (\ref{modzbig}), (\ref{eqC}) and (\ref{eqbeta}). It follows, by Theorem~\ref{webtheorem}, that $A_R(f)$ is a {\spw}.

The rest of the proof of Theorem~\ref{T1} is now quite straightforward. Since $A_R(f)$ is a {\spw}, $A(f)$ and $I(f)$ are also {\spws} by \cite[Theorem 1.4]{Rippon01102012}. Since, by Corollary~\ref{Cnomconn}, $f$ has no multiply connected Fatou components, the fact that each of $$A_R(f) \cap J(f), \ A(f) \cap J(f), \ I(f) \cap J(f), \text{ and } J(f)$$ is also a {\spw} follows by \cite[Theorem 1.5(a)]{Rippon01102012}.
\end{proof}
%
%
%
\section{On logarithmic tracts}
\label{Stracts}
In this brief section we give some definitions and one preliminary result, in order to prove that if $f \in \mathcal{E}_n$, for $n\geq 3$, then $f$ does not have a logarithmic tract. 

We use definitions taken from \cite{MR2439666}. Suppose that $f$ is a {\tef}. Suppose also that $U$ is an unbounded domain with unbounded complement, the boundary of which consists of piecewise smooth curves. We say that $U$ is a \itshape direct tract \normalfont of $f$ if there exists $R > 0$ such that $|f(z)| = R$, for $z \in \partial U$, and also $|f(z)| > R$, for $z \in U$. If, in addition, the restriction $f : U \to \{z : |z| > R\}$ is a universal covering, then we say that $U$ a \itshape logarithmic tract \normalfont of $f$. In particular, if $U$ is a logarithmic tract, then $U$ is simply connected.

Although Lemma~\ref{LnotinB} could be proved directly, it seems most straightforward to use the following result of Rippon and Stallard, which is part of \cite[Theorem 1.8]{Rippon01102012}.
\begin{lemma}
\label{RS1}
Let $f$ be a {\tef}, let $R > 0$ be such that $M(r,f) > r$, for $r \geq R$, and let $A_R(f)$ be a {\spw}. Then there is no path to infinity on which
$f$ is bounded, and $f \notin \mathcal{B}$.
\end{lemma}
\begin{lemma}
\label{LnotinB}
Suppose that $f \in \mathcal{E}_n$, for $n\geq 3$. Then $f$ does not have a logarithmic tract. In particular $f \notin \mathcal{B}$. 
\end{lemma}
\begin{proof}
Choose $R > 0$ such that $M(r,f) > r$, for $r \geq R$. By Theorem~\ref{T1} we have that $A_R(f)$ is a {\spw}. Suppose that $U$ is a direct tract of $f$. Then, by Lemma~\ref{RS1}, the boundary of $U$ has only bounded components. Hence $U$ cannot be simply connected, and so $U$ is not a logarithmic tract.
\end{proof}

%
%
%
Acknowledgment: \normalfont
The author is grateful to Phil Rippon and Gwyneth Stallard for all their help with this paper.
%
%
\bibliographystyle{acm}

\end{document}